\documentclass[a4paper,11pt]{article}
\usepackage[utf8]{inputenc}
\usepackage{amsmath,amsthm,amssymb,amsfonts}
\usepackage{mathtools}
\usepackage{color}
\usepackage{enumerate}
\usepackage{graphicx}
\usepackage{geometry}
\usepackage{framed}  
\usepackage[ruled,vlined]{algorithm2e}

\usepackage{hyperref}
\hypersetup{
	colorlinks=true,   
	linkcolor=blue,    
	citecolor=blue,    
	urlcolor=blue      
}

\newtheorem{theorem}{Theorem}[section]
\newtheorem{lemma}[theorem]{Lemma}

\newtheorem{assumption}[theorem]{Assumption}
\theoremstyle{remark}
\newtheorem{remark}[theorem]{Remark}
\newtheorem{example}[theorem]{Example}

\DeclareMathOperator{\prox}{prox}

\DeclareMathOperator{\zer}{zer}
\DeclareMathOperator{\Fix}{Fix}
\DeclareMathOperator{\range}{range}

\newcommand{\Hilbert}{\mathcal{H}}
\newcommand{\setto}{\rightrightarrows}

\DeclareMathOperator{\Id}{Id}
\DeclareMathOperator*{\argmin}{arg\,min}

\newcommand{\bz}{\mathbf{z}}
\newcommand{\bx}{\mathbf{x}}
\newcommand{\by}{\mathbf{y}}
\newcommand{\bv}{\mathbf{v}}

\newcommand{\bW}{\mathbf{W}}
\newcommand{\bS}{\mathbf{S}}
\newcommand{\bL}{\mathbf{L}}
\newcommand{\bM}{\mathbf{M}}

\newcommand{\bF}{\mathbf{F}}
\newcommand{\bN}{\mathbf{N}}

\title{Frugal and Decentralised Resolvent Splittings\\ Defined by Nonexpansive Operators}
\author{Matthew K. Tam\thanks{School of Mathematics \& Statistics,
                             The University of Melbourne,
                             Parkville VIC 3010, Australia.
                             Email:~\href{href:matthew.tam@unimelb.edu.au}
                                         {matthew.tam@unimelb.edu.au}}
}

\begin{document}
\maketitle

\begin{abstract}
Frugal resolvent splittings are a class of fixed point algorithms for finding a zero in the sum of the sum of finitely many set-valued monotone operators, where the fixed point operator uses only vector addition, scalar multiplication and the resolvent of each monotone operator once per iteration. In the literature, the convergence analyses of these schemes are performed in an inefficient, algorithm-by-algorithm basis. In this work, we address this by developing a general framework for frugal resolvent splitting which simultaneously covers and extends several important schemes in the literature. The framework also yields a new resolvent splitting algorithm which is suitable for decentralised implementation on regular networks.
\end{abstract}


\section{Introduction and Preliminaries}
Throughout this work, $\Hilbert$ denotes a real Hilbert space with inner-product $\langle\cdot,\cdot\rangle$ and induced norm~$\|\cdot\|$. We are interested in solving the \emph{monotone inclusion} given by
\begin{equation}\label{eq:zero}
  \text{find~}x\in\left(\sum_{i=1}^nF_i\right)^{-1}(0) \subseteq \Hilbert,
\end{equation}
where $F_1,\dots,F_n\colon\Hilbert\setto\Hilbert$ are (set-valued) \emph{maximally monotone} operators~\cite{bauschke2017convex,ryu2022large}. Recall that $F_i$ is said to be maximally monotone if it is \emph{monotone} in the sense
 $$ \langle x-\bar{x},y-\bar{y}\rangle \geq 0 \quad\forall y\in F_i(x),\,\forall\bar{y}\in F_i(\bar{x}) $$
and its graph contains no proper monotone extensions. The abstraction given by \eqref{eq:zero} covers many fundamental problems arising in mathematical optimisation involving objectives with finite-sum structure. Most notable are the following two types of  examples:
\begin{example}[nonsmooth minimisation]\label{ex:sum min}
Let $f_1,\dots,f_n\colon\Hilbert\to(-\infty,+\infty]$ be proper, lsc, convex functions. Consider the finite-sum nonsmooth minimisation problem given by
\begin{equation}\label{eq:sum min}
\min_{x\in\Hilbert}\sum_{i=1}^nf_i(x). 
\end{equation}
Assuming sufficient regularity so that the sum-rule holds, the first-order optimality condition for \eqref{eq:sum min} is given by the inclusion \eqref{eq:zero} with $F_i=\partial f_i$. Note that here ``$\partial$'' denotes the \emph{subdifferential} in the sense of convex analysis.
\end{example}
\begin{example}[nonsmooth min-max]\label{e:sum minmax}
 Let $\phi_1,\dots,\phi_n\colon\Hilbert_1\times\Hilbert_2\to [-\infty,+\infty]$ be proper, saddle-functions~\cite{rockafellar1970monotone}. That is, $\phi_i$ is convex-concave and there exists a point $(u,v)\in\Hilbert_1\times\Hilbert_2$ such that $\phi_i(\cdot,v)<+\infty$ and $\phi_i(u,\cdot)>-\infty$.  Consider the finite-sum nonsmooth minmax problem given by
\begin{equation}\label{eq:sum minmax}
\min_{u\in\Hilbert_1}\max_{v\in\Hilbert_2}\sum_{i=1}^n\phi_i(u,v) 
\end{equation}
Assuming sufficient regularity so that the sum-rule holds, the first-order optimality condition for \eqref{eq:sum minmax} is given by the inclusion \eqref{eq:zero} with $\Hilbert:=\Hilbert_1\times\Hilbert_2$, $x=(u,v)$ and $F_i:=\binom{\partial_u\phi_i}{\partial_v(-\phi_i)}$. Here ``$\partial_u$'' (resp.\ $\partial_v$) denotes the \emph{partial subdifferential} with respect to the variable $u$ (resp.\ $v$).
\end{example}

\paragraph{Frugal resolvent splittings.}
In this work, we consider \emph{frugal resolvent splitting} for solving~\eqref{eq:zero}.  Roughly speaking, a fixed point algorithm, defined by an operator $T\colon\Hilbert^d\to\Hilbert^d$, for solving \eqref{eq:zero} is a \emph{resolvent splitting} if it can be described using only vector addition, scalar multiplication, and the resolvents $J_{F_1}:=(\Id+F_1)^{-1},\dots,J_{F_n}:=(\Id+F_n)^{-1}$. A resolvent splitting is \emph{frugal} if it uses each of the resolvents only once per iteration. For details and extensions of these notions, see~\cite{ryu2020uniqueness,malitsky2021resolvent,morin2022frugal,aragon2022primal}.

In recent times, there has been interest in \emph{minimal} frugal resolvent splittings. As a consequence of~\cite[Theorem~3.3]{malitsky2021resolvent}, these are precisely the frugal resolvent splittings such that the operator $T\colon\Hilbert^d\to\Hilbert^d$ has $d=n-1$. In the case $n=2$, the \emph{Douglas--Rachford algorithm} is a minimal frugal resolvent splitting which is defined by $T\colon\Hilbert\to\Hilbert$ where
 $$ T(z) = z+\gamma\left(J_{F_2}\bigl(2J_{F_1}(z)-z\bigr)-J_{F_1}(z)\right). $$
For $n=3$, \emph{Ryu's splitting algorithm} \cite{aragon2021strengthened,ryu2020uniqueness} is a minimal resolvent splitting which is defined by $T\colon\Hilbert^2\to\Hilbert^2$ where
\begin{equation}\label{eq:ryu3 intro}
 T(\bz) = \bz + \gamma\begin{pmatrix}J_{F_3}\bigl(J_{F_1}(z_1)-z_1+J_{F_2}\bigl(J_{F_1}(z_1)+z_2\bigr)-z_2\bigr)-J_{F_1}(z_1)\\ J_{F_3}\bigl(J_{F_1}(z_1)-z_1+J_{F_2}\bigl(J_{F_1}(z_1)+z_2\bigr)-z_2\bigr)-J_{F_2}\bigl(J_{F_1}(z_1)+z_2\bigr)\end{pmatrix}. 
\end{equation}
As observed in \cite[Remark~4.7]{malitsky2021resolvent}, attempts to extends Ryu's splitting algorithm to $n>3$ operators have been so-far unsuccessful. To address this for $n\geq 3$, Malitsky \& Tam \cite{malitsky2021resolvent} introduced the minimal frugal resolvent splitting defined by $T_A\colon\Hilbert^{n-1}\to\Hilbert^{n-1}$ where
$$T(\bz)
:=  \bz +
\gamma\begin{pmatrix}
x_2-x_1 \\
x_3-x_2 \\
\vdots \\
x_{n}-x_{n-1} \\
\end{pmatrix}\text{~where~} \begin{cases}
    x_1=J_{F_1}(z_1),  \\
    x_i=J_{F_i}(z_i+x_{i-1}-z_{i-1}) \quad \forall i\in\{2,\dots,n-1\}, \\
    x_n=J_{F_n}(x_1+x_{n-1}-z_{n-1}). 
   \end{cases}
$$
Despite many similarities in the convergence proofs of the above algorithms (\emph{e.g.,} all of the underlying operators are nonexpansive), their convergence analyses are performed in an inefficient, algorithm-by-algorithm basis in the literature. In this work, we address this by providing a general framework for convergence of (not necessarily minimal) frugal resolvent splittings which simultaneously covers all of the aforementioned algorithm. This framework reduces nonexpansiveness of the underlying operator $T$ to a question of negative definiteness of a certain matrix defined by algorithm coefficients. In addition to significantly simplifying the proofs of \cite{ryu2020uniqueness,malitsky2021resolvent}, the framework  also suggests a convergent extension of Ryu's splitting algorithm for $n>3$ operators. More precisely, it can be used to show that the operator $T\colon\Hilbert^{n-1}\to\Hilbert^{n-1}$ is nonexpansive where
\begin{equation}\label{eq:ryun intro}
T(\bz)
=  \bz +
\gamma\sqrt{\frac{2}{n-1}}\begin{pmatrix}
x_n-x_1 \\
x_n-x_2 \\
\vdots \\
x_{n}-x_{n-1} \\
\end{pmatrix}
\text{~where~}
 \begin{cases}
    x_i=J_{F_i}\bigl(\sqrt{\frac{2}{n-1}}\,z_i+\frac{2}{n-1}\displaystyle\sum_{j=1}^{i-1}x_j\bigr) \quad \forall i\in\{1,\dots,n-1\} \\
    x_n=J_{F_n}\bigl(\frac{2}{n-1}{\displaystyle\sum_{j=1}^{n-1}x_i}-\sqrt{\frac{2}{n-1}}\displaystyle\sum_{i=1}^{n-1}z_i\bigr). 
   \end{cases}
\end{equation}
Note that \eqref{eq:ryun intro} coincides with \eqref{eq:ryu3 intro} when $n=3$.

\paragraph{Decentralised algorithms.}
Another focus  of this work is on \emph{distributed algorithms}~\cite{bertsekas2015parallel} for solving~\eqref{eq:zero} under the following conditions:
\begin{enumerate}[(a)]
\item Each operator $F_i$ in \eqref{eq:zero} is known by single node in a graph $G=(V,E)$ with $|V|=n$. More precisely, $F_i$ is known by the $i$th node.
\item Each node has its own local variables, and can compute standard arithmetic operations as well as the \emph{resolvent} of its monotone operator $F_i$ given by $J_{F_i}:=(\Id+F_i)^{-1}$ (see \cite[Section~23]{bauschke2017convex}).
\item Two nodes in $G$ can communicate the results of their computations directly with one another only if they are connected by an edge in $G$.
\item The algorithm is ``decentralised'' in the sense that its does not necessarily require the computation of a global sum across all local variables of the nodes.
\end{enumerate}

Stated at the abstraction of monotone inclusions, the classical way to solve \eqref{eq:zero} (without distributed considerations) involves using the \emph{Douglas--Rachford algorithm} applied to a reformulation in the product-space $\Hilbert^n$ (see, for instance, \cite[Example~2.7]{malitsky2021resolvent} or \cite[Proposition~26.12]{bauschke2017convex}). Given an initial point $\bz^0=(z_1^0,\dots,z_n^n)\in\Hilbert^n$ and a constant $\gamma\in(0,2)$, this approach gives rise to the iteration given by
\begin{equation}\label{eq:dr product}
\left\{\begin{aligned}
  x^k       &= \frac{1}{n}\sum_{j=1}^nz^k_j \\
  z^{k+1}_i &= z^k_i + \gamma\bigl(J_{F_i}\bigl(2x^k -z^k_i\bigr)-x^k\bigr) &&\forall i\in\{1,\dots,n\}.
\end{aligned}\right.
\end{equation}
In terms of distributed computing, this iteration requires the global sum across the local variables $z_1^k,\dots,z_n^k$ to compute the aggregate $x^k$ and so it does not satisfy requirements (c) or (d) above. On the other hand, the $z_i$-update in \eqref{eq:dr product} does satisfy properties (a) and (b), assuming it is computed by node $i$. For recent variations on \eqref{eq:dr product}, the reader is referred to \cite{campoy2022product,condat2019proximal}.

In order to describe other algorithms in the literature, it is convenient to introduce some notation. Given a matrix $W=(w_{ij})\in\mathbb{R}^{p\times l}$, we denote $\bW:=W\otimes \Id$ were $\Id$ denotes the identity operator on $\Hilbert$ and $\otimes$ denotes the Kronecker product. In other words, $\bW$ can be identified with the bounded linear operator from $\Hilbert^l$ to $\Hilbert^p$ given by
 $$ \bW = W\otimes \Id = \begin{bmatrix}
                          w_{11}\Id & w_{12}\Id & \dots & w_{1l}\Id \\
                          w_{21}\Id & w_{22}\Id & \dots & w_{2l}\Id \\                          
                          \vdots & \vdots & \ddots & \vdots \\
                          w_{p1}\Id & w_{p2}\Id & \dots & w_{pl}\Id \\                          
                         \end{bmatrix}.$$
Given maximally monotone operators $F_1,\dots,F_n\colon\Hilbert^n\setto\Hilbert$, we denote the maximally monotone operator $\bF\colon\Hilbert^n\to\Hilbert^n$ by
 $ \bF(\bx) := \bigl( F_1(x_1),\dots,F_n(x_n) \bigr)$ for all $\bx=(x_1,\dots,x_n)\in\Hilbert^n.$ Consequently, its resolvent $J_{\bF}\colon\Hilbert^n\to\Hilbert^n$ is given by  $J_{\bF} = \bigl( J_{F_1},\dots,J_{F_n} \bigr). $ Recall also that, when $F_i=\partial f_i$ where of $f_i$ is proper, lsc and convex, then the resolvent coincides with \emph{proximity operator} \cite[Section~24]{bauschke2017convex} given 
 $$ J_{F_i}=\prox_{f_i}=\argmin_{y\in\Hilbert}\left\{f_i(y)+\frac{1}{2}\|\cdot-y\|^2\right\}. $$                      
                    
Now, in the special case of Example~\ref{ex:sum min} where $J_{\partial f_i}=\prox_{f_i}$, the \emph{proximal exact first-order algorithm (P-EXTRA)}~\cite{shi2015proximal} can be used to solve \eqref{eq:zero}. Given $\bx^0\in\Hilbert^n$ and initialising with $\by^0=\bW\bx^0$ and $\bx^1=\prox_{\alpha f}(\by^0)$, this method iterates according to
\begin{equation}\label{eq:extra}
\left\{\begin{aligned}
   \by^{k} &= \bW\bx^k+\by^{k-1}-\widetilde{\bW}\bx^{k-1} \\
   \bx^{k+1}           &= \prox_{\alpha f}(\by^{k}). 
  \end{aligned}\right. 
\end{equation}
where $W,\widetilde{W}\in\mathbb{R}^{n\times n}$ are ``mixing matrices'' satisfying (c)  that describe the communication between nodes. For specific examples, the reader is referred to \cite[Assumption~1]{shi2015proximal}. This method satisfies conditions (a)-(d) in the minimisation setting with $x_i,y_i,z_i$ being the local variables of the $i$th node.

Let $L\in\mathbb{R}^{n\times n}$ denote the \emph{graph Laplacian} of $G=(V,E)$ \cite{chung1996lectures} where $|V|=n$. For the general monotone inclusion~\eqref{eq:zero}, the \emph{primal dual hybrid gradient (PDHG) algorithm} \cite{chambolle2011first,condat2013primal} can be used. After a change of variables (see  \cite[Section~5.1]{malitsky2021resolvent}), the PDHG can be expressed as
\begin{equation}\label{eq:pdhg}
    \left\{\begin{aligned}
     \bx^{k+1} &= J_{\tau\bF}\big(\bx^k-\tau\bv^k) \\
     \bv^{k+1} &= \bv^k + \sigma L\bigl(2\bx^{k+1}-\bx^{k}\bigr) \\
    \end{aligned}\right.
\end{equation}    
where the stepsizes are required to satisfy $\tau\sigma\|L\|\leq 1$. This method satisfies conditions (a)-(d) with $x_i,v_i$ being the variables local to the $i$th node.

Using our framework for convergence of frugal resolvent splittings, this work derives new fixed point algorithms for \eqref{eq:zero} which are suitable to decentralised implementation, as outlined in conditions (a)-(d). Given $\gamma\in(0,1)$ and $\tau:=\frac{n}{|E|}$, the most similar realisation of our approach to~\eqref{eq:pdhg} (see Example~\ref{ex:regular}) is given by
\begin{equation}\label{eq:d reg full}
\left\{ \begin{aligned}
  x^k_i &= J_{F_i}\bigl(v^k_i+\tau\sum_{j=1}^{i-1}A_{ij}x_j^k\bigr) \quad\forall i\in\{1,\dots,n\} \\
  v^{k+1}_i &= v^k_i-\gamma\tau\sum_{j=1}^nL_{ij}x^k_j
     \end{aligned}\right. 
\end{equation}     
where $A\in\mathbb{R}^{n\times n}$ denotes the adjacency matrix of the graph $G$. Note that this method satisfies conditions (a)-(d) with $x_i,v_i$ be the variable local to the $i$th node. By letting $N\in\mathbb{R}^{n\times n}$ denote the lower triangular part of $A$, the iteration \eqref{eq:d reg full} can be compactly expressed as
\begin{equation}\label{eq:d reg alg}
 \left\{\begin{aligned}
 \bx^k &=J_{\bF}(\bv^k+\tau\bN\bx^k) \\
 \bv^{k+1} &= \bv^k-\gamma\tau\bL\bx^k.
 \end{aligned}\right.
\end{equation}
Here we note that $\bx^k$ is well-defined due to the fact that $N$ is lower triangular with zeros on the diagonal. A feature of \eqref{eq:d reg alg}, which is not present in \eqref{eq:extra} or \eqref{eq:pdhg}, is that its $x_i$-updates in \eqref{eq:d reg alg} are performed sequentially within each iteration. Furthermore, the update \eqref{eq:d reg alg} requires no knowledge of past iterates.

\paragraph{Structure of this Work.} The remainder of the work is structured as follows. In Section~\ref{s:resolvent splitting}, we develop an abstract framework for convergence of frugal resolving splitting algorithms as fixed point iterations, which extend ideas from \cite{malitsky2021resolvent}, and analyse its convergence properties. More precisely, the main results in Section~\ref{s:resolvent splitting} establish averaged nonexpansivity of the algorithmic operator (Lemma~\ref{lem:nonexpansive}) and weak convergence of its iterates (Theorem~\ref{th:resolvent splitting}). Then, in Section~\ref{s:realisations}, we provide various concrete realisation of the abstract framework. In particular, this includes a previously unknown extension of Ryu's method~\cite{ryu2020uniqueness} to finitely many operators (Example~\ref{ex:ryu new}), the method proposed in \cite{malitsky2021resolvent} (Example~\ref{ex:minimal}), as well as the new decentralised scheme~\eqref{eq:d reg alg} when the communication graph $G$ is regular (Example~\ref{ex:regular}).
Numerical examples to demonstrated the proposed framework are reported in Section~\ref{s:experiments}.

\section{A Framework for Resolvent Splitting}\label{s:resolvent splitting}
In this section, we consider an operator $T\colon\Hilbert^m\to\Hilbert^m$ of the form
\begin{equation}\label{eq:T}
 T(\bz) := \bz + \gamma \bM\bx,\qquad \by=\bS\bz+\bN\bx,\qquad \bx=J_{\bF}(\by), 
\end{equation}
where $M\in\mathbb{R}^{m\times n}, S\in\mathbb{R}^{n\times m}$ and $N\in\mathbb{R}^{n\times n}$ are coefficient matrices. Moreover, we assume that the resolvents are evaluated in the order $J_{A_1},\dots,J_{A_n}$, which means that the matrix $N$ is necessarily lower triangular with zeros on the diagonal. Altogether, we require the following assumptions on the coefficient matrices.

\begin{assumption}\label{ass:resolvent}
The coefficient matrices $M\in\mathbb{R}^{m\times n}, S\in\mathbb{R}^{n\times m}$ and $N\in\mathbb{R}^{n\times n}$ satisfy the following properties.
\begin{enumerate}[(a)]
\item\label{ass:res_a} $\ker M\subseteq\mathbb{R}e$ where $e=(1,\dots,1)^T\in\mathbb{R}^n$.
\item\label{ass:res_b} $\sum_{i,j=1}^nN_{ij}=n$ and $N$ is lower triangular with zeros on the diagonal.
\item\label{ass:res_c} $S^*=-M$.
\item\label{ass:res_d} $M^*M+N+N^*-2\Id\preceq 0. $
\end{enumerate}
\end{assumption}

\begin{remark}
When Assumption~\ref{ass:resolvent} holds, properties of the Kronecker product imply that
\begin{equation}\label{eq:ker M}
 \ker\bM = \Delta := \{\bx=(x_1,\dots,x_n)\in\Hilbert^n:x_1=\dots=x_n\}, 
\end{equation}
$\bS^*=-\bM$ and $\bM^*\bM+\bN+\bN^*-2\Id\preceq 0$. To see \eqref{eq:ker M}, we note the following chain of implications
\begin{align*}
 (x_1,\dots,x_n)\in\ker\bM 
 &\implies (\langle x_1,v\rangle,\dots,\langle x_n,v\rangle)\in \ker M\quad \forall v\in\Hilbert\\
 &\implies \langle x_1,v\rangle=\dots=\langle x_n,v\rangle\quad \forall v\in\Hilbert \\
 &\implies x_1=\dots=x_n \\
 &\implies (x_1,\dots,x_n)\in\ker\bM.
\end{align*}
\end{remark}

\begin{lemma}[Fixed points and zeros]\label{lem:fixed points}
Suppose Assumptions~\ref{ass:resolvent}\eqref{ass:res_a}-\eqref{ass:res_c} holds, and denote
$$ \Omega := \{(\bz,x)\in\Hilbert^m\times\Hilbert\colon \bx=J_{\bF}(\bS\bz+\bN\bx)\text{~where~}\bx=(x,\dots,x)\in\Delta\}. $$
Then the following assertions hold.
\begin{enumerate}[(a)]
\item If $\bz\in\Fix T$, then there exists $x\in\Hilbert$ such that $(\bz,x)\in\Omega$.
\item If $x\in\zer\left(\sum_{i=1}^nF_i\right)$, then there exists $\bz\in\Hilbert^m$ such that $(\bz,x)\in\Omega$.
\item If $(\bz,x)\in\Omega$, then $\bz\in\Fix T$ and $x\in\zer(\sum_{i=1}^nF_i)$. 
\end{enumerate}
Consequently, 
 $$\Fix T\neq\emptyset\iff\Omega\neq\emptyset\iff \zer\left(\sum_{i=1}^nF_i\right)\neq\emptyset.$$
\end{lemma}
\begin{proof}
(a):~Let $\bz\in\Fix T$. Set $\by:=\bS\bz+\bN\bx$ and $\bx:=J_{\bF}(\by)$. Since $\bz\in\Fix T$, we have $\bx\in\ker M$ and so Assumption~\ref{ass:resolvent}\eqref{ass:res_a} implies $\bx\in\Delta$. 

(b):~Let $x\in\zer\bigl(\sum_{i=1}^nF_i\bigr)$ and set $\bx:=(x,\dots,x)\in\Delta$. Then there exists $\bv=(v_1\dots,v_n)\in \bF(\bx)$ such that $\sum_{i=1}^nv_i=0$. Define $\by=(y_1,\dots,y_n)\in\Hilbert^n$ according to $\by:=\bv+\bx$, so that $\bx=J_{\bF}(\by)$. To complete the proof, we must show there exists $\bz\in\Hilbert^m$ such that $\by=\bS\bz+\bN\bx$, which is equivalent to the $\by-\bN\bx\in\range\bS$. 
To this end, first note that Assumption~\ref{ass:resolvent}\eqref{ass:res_a} \& \eqref{ass:res_c} imply
 $$ \range \bS = (\ker \bS^*)^\perp = (\ker \bM)^\perp \supseteq \Delta^\perp = \{(x_1,\dots,x_n)\in\Hilbert^n:\sum_{i=1}^nx_i=0\}. $$
 With the help of Assumption~\ref{ass:resolvent}\eqref{ass:res_b}, we see that
$$ \sum_{i=1}^n(\by-\bN\bx)_i = \sum_{i=1}^ny_i - \sum_{i,j=1}^nN_{ij}x =\left(\sum_{i=1}^nv_i+nx\right) - nx =0, $$
Hence $\by-\bN\bx\in\range \bS$, as required.

(c):~Let $(\bz,x)\in\Omega$ and set $\by:=\bS\bz+\bN\bx$ where $\bx=(x,\dots,x)\in\Delta$. Since $\bx=J_{\bF}(\by)\in\ker\bM$, we have $\bz\in\Fix T$ by Assumption~\ref{ass:resolvent}\eqref{ass:res_a}. Denote $e=(1,\dots,1)^\top\in\mathbb{R}^n$. Then $Me=0$ by Assumption~\ref{ass:resolvent}\eqref{ass:res_a}. Since $\bF(\bx)\ni \by-\bx=-\bM^*\bz+\bN\bx-\bx$, Assumption~\ref{ass:resolvent}\eqref{ass:res_b} gives
$$ \sum_{i=1}^nF_i(x) \ni -\bigl((e^\top M^*)\otimes \Id\bigr) \bz + \sum_{i,j=1}^nN_{ij}x - nx = - \bigl((Me)^\top\otimes \Id\bigr)\bz+nx-nx = 0, $$
This shows that $x\in\zer\bigl(\sum_{i=1}^nF_i\bigr)$ and so completes the proof.
\end{proof}

Recall that an operator $S$ is \emph{$\alpha$-averaged nonexpansive} if $\alpha\in(0,1)$ and, for all $x,y$, we have
 $$ \|S(x)-T(y)\|^2 + \frac{1-\alpha}{\alpha}\|(\Id-S)(x)-(\Id-S)(y)\|^2 \leq \|x-y\|^2. $$
 
\begin{lemma}[Nonexpansivity]\label{lem:nonexpansive}
Let $F_1,\dots,F_n\colon\Hilbert\setto\Hilbert$ be maximally monotone.
Suppose Assumptions~\ref{ass:resolvent}\eqref{ass:res_c}~\&~\eqref{ass:res_d} hold, and let $\gamma>0$. Then, for all $\bz,\bar{\bz}\in\Hilbert^m$, we have
\begin{multline}\label{eq:averaged}
 \|T(\bz)-T(\bar{\bz})\|^2 + \frac{1-\gamma}{\gamma}\|(\Id-T)(\bz)-(\Id-T)(\bar{\bz})\|^2
 \\
  + \gamma\langle \bx-\bar{\bx},[2\Id-\bM^*\bM-\bN-\bN^*](\bx-\bar{\bx})\rangle \leq \|\bz-\bar{\bz}\|^2.
\end{multline}
where $\bx=J_{\bF}(\bS\bz+\bN\bx)$ and $\bar{\bx}=J_{\bF}(\bS\bar{\bz}+\bN\bar{\bx})$.
In particular, $T$ is $\gamma$-averaged nonexpansive whenever $\gamma\in(0,1)$.
\end{lemma}
\begin{proof}
Let $\bz,\bar{\bz}\in\Hilbert^m$, Set $\bx:=J_{\bF}(\by)$ where $\by:=\bS\bz+\bN\bx$ and $\bar{\bx}:=J_{\bF}(\bar{\by})$ where $\bar{\by}:=\bS\bar{\bz}+\bN\bar{\bx}$. Since $\by-\bx\in \bF(\bx)$ and $\bar{\by}-\bar{\bx}\in\bF(\bar{\bx})$, monotonicity of $\bF=(F_1,\dots,F_n)$ gives
\begin{equation}\label{eq:mono}
\begin{aligned}
0 &\leq \langle \bx-\bar{\bx},(\by-\bx)-(\bar{\by}-\bar{\bx})\rangle \\
  &= \langle \bx-\bar{\bx},(\bS\bz+\bN\bx-\bx)-(\bS\bar{\bz}+\bN\bar{\bx}-\bar{\bx})\rangle \\
  &= \langle \bS^*\bx-\bS^*\bar{\bx},\bz-\bar{\bz}\rangle + \langle \bx-\bar{\bx},(\bN-\Id)\bx-(\bN-\Id)\bar{\bx}\rangle.
\end{aligned}
\end{equation}
By Assumption~\ref{ass:resolvent}\eqref{ass:res_c}, $\bS^*=-\bM$ and so the first term in \eqref{eq:mono} can be expressed as
\begin{equation}\label{eq:adj}
\begin{aligned}
&\langle \bS^*\bx-\bS^*\bar{\bx},\bz-\bar{\bz}\rangle \\
  &=\langle (-\bM\bx)-(-\bM\bar{\bx}),\bz-\bar{\bz}\rangle \\
  &= \frac{1}{\gamma}\langle (\Id-T)(\bz)-(\Id-T)(\bar{\bz}),\bz-\bar{\bz}\rangle  \\    
  &= \frac{1}{2\gamma}\left(\|\bz-\bar{\bz}\|^2+\|(\Id-T)(\bz)-(\Id-T)(\bar{\bz})\|^2-\|T(\bz)-T(\bar{\bz})\|^2\right).
\end{aligned}
\end{equation}
The second term in \eqref{eq:mono} can be expressed as
\begin{equation}\label{eq:nsd}
\begin{aligned}
 &\langle \bx-\bar{\bx},(\bN-\Id)\bx-(\bN-\Id)\bar{\bx}\rangle \\
 &= \frac{1}{2}\langle \bx-\bar{\bx},[\bM^*\bM+2\bN-2\Id](\bx-\bar{\bx})\rangle - \frac{1}{2}\|\bM\bx-\bM\bar{\bx}\|^2 \\
  &= \frac{1}{2}\langle \bx-\bar{\bx},[\bM^*\bM+\bN+\bN^*-2\Id](\bx-\bar{\bx})\rangle -\frac{1}{2\gamma^2}\|(\Id-T)(\bz)-(\Id-T)(\bar{\bz})\|^2.
\end{aligned}
\end{equation}
Substituting \eqref{eq:adj} and \eqref{eq:nsd} into \eqref{eq:mono}, followed by multiplying by $2\gamma$, gives \eqref{eq:averaged}. In particular, if Assumption~\ref{ass:resolvent}\eqref{ass:res_d} holds, then the inner-product on the LHS of \eqref{eq:averaged} is non-negative and hence $T$ is $\gamma$-averaged nonexpansive whenever $\gamma\in(0,1)$.
\end{proof}

\begin{remark}
As we will explain in Section~\ref{s:realisations}, the setting considered here includes Ryu's method~\cite{ryu2020uniqueness} and the resolvent splitting with minimal lifting from \cite{malitsky2021resolvent} as special cases. Consequently,  Lemma~\ref{lem:nonexpansive} provides a very concise proof covering \cite[Section~4.2]{ryu2020uniqueness} and \cite[Lemma~4.3]{malitsky2021resolvent} (each of which took several pages of algebraic manipulation).
\end{remark}

The following theorem is our main result regarding convergence of the fixed point iteration corresponding to $T$ in \eqref{eq:T}.
\begin{theorem}\label{th:resolvent splitting}
Let $F_1,\dots,F_n\colon\Hilbert\setto\Hilbert$ be maximally monotone operators with $\zer(\sum_{i=1}^nF_i)\neq\emptyset$. Suppose Assumption~\ref{ass:resolvent} holds and let $\gamma\in(0,1)$. Let $\bz^0\in\Hilbert$ and define the sequences $(\bz^k)$ and $(\bx^k)$ according to
\begin{equation}\label{eq:th iter}
 \bz^{k+1} = T(\bz^k) = \bz^k+\gamma\bM\bx^k,\qquad \bx^k=J_{\bF}(\bS\bz^k+\bN\bx^k). 
\end{equation}
Then the following assertions hold.
\begin{enumerate}[(a)]
\item As $k\to\infty$, we have $\bz^k-\bz^{k+1}\to 0$ and $\sum_{i=1}^ns_ix^k_i\to 0$ for all $s\in\mathbb{R}^n$ with $\sum_{i=1}^ns_i=0$.
\item The sequence $(\bz^k)$ converges weakly to a point $\bar{\bz}\in\Fix T$.
\item The sequence $(\bx^k)$ converges weakly to a point $(\bar{x},\dots,\bar{x})\in\Hilbert^n$ such that $\bar{x}\in\zer(\sum_{i=1}^nF_i)$.
\end{enumerate}
\end{theorem}
\begin{proof}
(a)~\&~(b):~Since $\zer(\sum_{i=1}^nF_i)\neq\emptyset$, Lemma~\ref{lem:fixed points} implies $\Fix T\neq\emptyset$. By Lemma~\ref{lem:nonexpansive}, the operator $T$ is averaged nonexpansive and so it follows from \cite[Proposition~5.16]{bauschke2017convex} that $\bz^k-\bz^{k+1}\to 0$ and $(\bz^k)$ converges weakly to a point $\bar{\bz}\in\Fix T$.
Next, let $s\in\{s\in\mathbb{R}^n:\sum_{i=1}^ns_i=0\}=\Delta^\perp \subseteq (\ker M)^\perp = \range M^* = \range S$. Then, there exists $v\in\mathbb{R}^m$ such that $s=Sv=-M^*v$ and hence
$$ - \sum_{i=1}^ns_ix_i^k = -(s^\top\otimes \Id)\bx^k = \bigl((v^\top M)\otimes \Id\bigr)\bx^k = (v^\top\otimes \Id)\bM\bx^k = \frac{1}{\gamma} (v^\top\otimes \Id)(\bz^{k+1}-\bz^k)\to 0.$$
(c):~Denote $\by^k=\bS\bz^k+\bN\bz^k$. We claim that the sequence $(\bx^k)$ is bounded. To see this, let $S_1\in\mathbb{R}^{1\times n}$ denote the first row of the matrix $S$ and let $N_1=(0,\dots,0)\in\mathbb{R}^{1\times n}$ denote the first row of the matrix $N$ (which is zero by Assumption~\ref{ass:resolvent}\eqref{ass:res_b}). Thus, by definition, we have
 $$ x^k_1 = J_{F_1}(y_1^k) = J_{F_1}(\bS_1\bz^k+\bN_1\bx^k) = J_{F_1}(\bS_1\bz^k). $$
By Lemma~\ref{lem:fixed points}, $(\bar{\bz},\bar{x})\in\Omega$ where $\bar{x}=J_{F_1}(\bS_1\bar{\bz})$. Thus, by nonexpansivity of resolvents \cite[Proposition~23.8]{bauschke2017convex} and boundedness of $(\bz^k)$, it then follows that $(x^k_1)$ is also bounded. By (a), it follows that $(\bx^k)$ is bounded, as claimed.

Next, using the identity $\bx^k=J_{\bF}(\by^k)$, followed by the identity $\by^k=\bS\bz^k+\bN\bx^k$ and the fact that $\range\bS \subseteq \Delta^\perp$, we deduce that
\begin{equation}\label{eq:demiclosed}
\mathcal{S}\left(\begin{pmatrix}
y_1^k-x_1^k \\ \vdots \\ y_{n-1}^k-x_{n-1}^k \\ x_n^k \\
\end{pmatrix}\right) \ni \begin{pmatrix}
x_1^k-x_n^k\\ \vdots \\ x_{n-1}^k-x_n^k \\ \sum_{i=1}^{n}y_i^k-\sum_{i=1}^nx_i^k \\
\end{pmatrix}
=\begin{pmatrix}
x_1^k-x_n^k\\ \vdots \\ x_{n-1}^k - x_n^k \\ \sum_{i=1}^{n}\bigl(\sum_{j=1}^nN_{ij}-1\bigr)x_i^k \\
\end{pmatrix}
\end{equation}
where $\mathcal{S}\colon\Hilbert^n\setto\Hilbert^n$ denotes the operator given by 
$$ \mathcal{S} := \begin{pmatrix}
(F_1)^{-1} \\ \vdots \\ (F_{n-1})^{-1} \\ F_n \\
\end{pmatrix}+\begin{bmatrix}
0 & \dots & 0 & -\Id \\ 
\vdots & \ddots & \vdots & \vdots \\
0 & \dots & 0 & -\Id \\
\Id & \dots & \Id & 0 \\
\end{bmatrix}. $$
Note that $\mathcal{S}$ is maximally monotone operator as the sum of the two maximally monotone operators, the latter having full domain \cite[Corollary~24.4(i)]{bauschke2017convex}. Consequently, its graph is sequentially closed in the weak-strong topology \cite[Proposition~20.32]{bauschke2017convex}. Note also that the RHS of \eqref{eq:demiclosed} converges strongly to zero as a consequence of (a) and Assumption~\ref{ass:resolvent}\eqref{ass:res_b}.

Let $\bx\in\Hilbert^n$ be an arbitrary weak cluster point of the sequence $(\bx^k)$. As a consequence of (a), it follows that $\bx=(x,\dots,x)$ for some $x\in\Hilbert$. Denote $\by=\bS\bar{\bz}+\bN\bx$. Taking the limit in \eqref{eq:demiclosed} along a subsequence of $(\bx^k)$ which converges weakly to $\bx$ gives
$$ \mathcal{S}\left(\begin{pmatrix}
y_1-x \\ \vdots \\ y_{n-1}-x \\ x \\
\end{pmatrix}\right) \ni \begin{pmatrix}
0\\ \vdots \\ 0 \\ 0 \\
\end{pmatrix} \quad\implies\quad \left\{\begin{aligned}
  F_i(x) &\ni y_i-x \quad \forall i\in\{1,\dots,n-1\}\\
  F_n(x) &\ni (n-1)x - \sum_{i=1}^{n-1}y_i. \\
\end{aligned}\right. $$
Altogether, this shows that $x=J_{F_1}(y_1)=J_{F_1}(S_1\bar{\bz})=\bar{x}$. In other words, $\bx=(\bar{x},\dots,\bar{x})$ is the unique weak sequential cluster point of the bounded sequence $(\bx^k)$. We therefore deduce that $(\bx^k)$ converges weakly to $\bx$, which completes the proof. 
\end{proof}

\begin{remark}\label{r:implementation}
A closer look at the iteration \eqref{eq:th iter} shows that, although convergence is analysed with respect to the sequence $(\bz^k)$, this sequence does not need to be directly computed. Indeed, using the fact that $\bS=-\bM^*$ and making the change of variables $\bv^k=\bS\bz^k$, we may write \eqref{eq:th iter} as 
 $$ \bv^{k+1} = \bv^k-\gamma \bM^*\bM\bx^k,\qquad \bx^{k} = J_{\bF}\bigl(\bv^k+\bN\bx^k\bigr). $$ 
In the case where $m>n$, this transformation also has the advantage of replacing a variable $\bz^k\in\Hilbert^m$ in the larger space with a variable $\bv^k\in\Hilbert^n$ in the smaller space.
\end{remark}

\section{Realisations of the Framework}\label{s:realisations}
In what follows, we provide several realisations of Theorem~\ref{th:resolvent splitting} for solving the monotone inclusion~\eqref{eq:zero}. Examples~\ref{ex:dr}, \ref{ex:ryu3} and \ref{ex:minimal} are already known in the literature. However, Examples~\ref{ex:ryu new} and \ref{ex:regular} are new.\footnote{An extension of Ryu splitting equivalent to \eqref{eq:ryu extension} in Example~\ref{ex:ryu new} was independently discovered using a different approach in \cite[p.~11]{bredies2022graph}. The original versions of the two preprints appeared on arXiv on the same day \cite{bredies2022graph,tam2022framework}.}

\begin{example}[Douglas--Rachford splitting]\label{ex:dr}
Douglas--Rachford splitting \cite{eckstein1992douglas,svaiter2011weak,bauschke2017douglas,bui2020douglas} for \eqref{eq:zero} with $n=2$ is defined by the operator $T\colon\Hilbert\to\Hilbert$ given by
 $$ T(z) = z + \gamma(x_2-x_1)\text{~~where~~}\begin{cases}
                                      x_1 = J_{F_1}(z) \\
                                      x_2 = J_{F_2}(2x_1-z)\\
                                    \end{cases}. $$
Thus, the corresponding coefficient matrices $M\in\mathbb{R}^{1\times 2},S\in\mathbb{R}^{2\times 1}$ and $N\in\mathbb{R}^{2\times 2}$ are given by
 $$ M = -S^* = \begin{bmatrix}
         -1 & 1 \\
        \end{bmatrix},\qquad
    N = \begin{bmatrix}
         0 & 0 \\
         2 & 0 \\
        \end{bmatrix}. $$            
Here, we have
 $$ M^*M + N + N^* - 2\Id = \begin{bmatrix}
 -1 & 1 \\
  1 & -1 \\
 \end{bmatrix}\preceq 0. $$        
\end{example}

\begin{example}[Ryu splitting for $n=3$]\label{ex:ryu3}
Ryu's splitting algorithm~\cite{ryu2020uniqueness,aragon2021strengthened} for \eqref{eq:zero} with $n=3$ is defined by the operator $T\colon\Hilbert^2\to\Hilbert^2$ given by
\begin{equation*}
 T(\bz) = \bz + \gamma\begin{pmatrix}
                     x_3-x_1 \\
                     x_3-x_2 \\
                   \end{pmatrix}\text{~~\color{black}where~~}\left\{\begin{aligned}
                                                 x_1 &= J_{F_1}(z_1) \\
                                                 x_2 &= J_{F_2}(z_2+x_1) \\
                                                 x_3 &= J_{F_3}(x_1-z_1+x_2-z_2). \\
                                                \end{aligned}\right.
\end{equation*}
Thus, the corresponding coefficient matrices $M\in\mathbb{R}^{2\times 3},S\in\mathbb{R}^{3\times 2}$ and $N\in\mathbb{R}^{3\times 3}$ are given by
\begin{equation}\label{eq:ryu3}
M = -S^* = \begin{bmatrix}
         -1 & 0 & 1 \\ 0 & -1 & 1 \\   
        \end{bmatrix},     
        \qquad
    N = \begin{bmatrix}
         0 & 0 & 0 \\ 1 & 0 & 0 \\ 1 & 1 & 0 \\   
        \end{bmatrix}.
\end{equation}        
It is straightforward to verify that Assumption~\ref{ass:resolvent}\eqref{ass:res_a}-\eqref{ass:res_c} holds. To verify Assumption~\ref{ass:resolvent}\eqref{ass:res_d}, observe that
$$ M^*M+N+N^*-2\Id = \begin{bmatrix}
         -1 & 1 & 0 \\ 1 & -1 & 0 \\  0 & 0 & 0 \\  
        \end{bmatrix}\preceq 0. $$       
\end{example}

\begin{example}[Resolvent splitting with minimal lifting]\label{ex:minimal}
The resolvent splitting algorithm considered in~\cite{malitsky2021resolvent,aragon2022distributed} for \eqref{eq:zero} with $n\geq 2$ is defined by the operator $T\colon\Hilbert^{n-1}\to\Hilbert^{n-1}$ given by
\begin{equation*}\label{eq:T minimal}
T(\bz)
=  \bz +
\gamma\begin{pmatrix}
x_2-x_1 \\
x_3-x_2 \\
\vdots \\
x_{n}-x_{n-1} \\
\end{pmatrix}
\text{~where~}
 \begin{cases}
    x_1=J_{F_1}(z_1),  \\
    x_i=J_{F_i}(z_i+x_{i-1}-z_{i-1}) \qquad \forall i\in\{2,\dots,n-1\} \\
    x_n=J_{F_n}(x_1+x_{n-1}-z_{n-1}). 
   \end{cases}
\end{equation*}
Thus, the coefficient matrices $M\in\mathbb{R}^{(n-1)\times n},S\in\mathbb{R}^{n\times(n-1)}$ and $N\in\mathbb{R}^{n\times n}$ are given by
$$ M = -S^* = \begin{bmatrix}
-1 &  1 & 0 & \dots & 0 & 0 \\
 0 & -1 & 1 & \dots & 0 & 0 \\
 0 & 0 & -1 & \dots & 0 & 0 \\ 
 \vdots & \vdots & \vdots & \ddots & \vdots & \vdots  \\
 0 &  0 & 0 & \dots & -1 & 1 \\ \end{bmatrix},\qquad
N = \begin{bmatrix}
0 & 0 & 0 & \dots & 0 & 0 \\
1 & 0 & 0 & \dots & 0 & 0 \\
0 & 1 & 0 & \dots & 0 & 0 \\
 \vdots & \vdots & \vdots & \ddots & \vdots & \vdots \\
0 & 0 & 0 & \dots & 0 & 0 \\
1 & 0 & 0 & \dots & 1 & 0 \\
\end{bmatrix}. $$
It is straightforward to verify that Assumption~\ref{ass:resolvent}\eqref{ass:res_a}-\eqref{ass:res_c} holds. To verify Assumption~\ref{ass:resolvent}\eqref{ass:res_d}, observe that
\begin{equation}\label{eq:X}
 M^*M+N+N^*-2\Id = \begin{bmatrix}
         -1 & 0 & \dots & 0 & 1 \\ 
          0 & 0 & \dots & 0 & 0 \\
          \vdots & \vdots & \ddots & \vdots & \vdots \\
          0 & 0 & \dots & 0 & 0 \\          
          1 & 0 & \dots & 0 & -1 \\
        \end{bmatrix}\preceq 0.
\end{equation}        
\end{example}

\begin{example}[An extension of Ryu splitting for $n\geq 3$ operators]\label{ex:ryu new}
As explained in \cite[Remark~4.7]{malitsky2021resolvent}, the naïve extension of Ryu splitting to $n>3$ operators fails to converge. However, the more general perspective offered by our framework suggests an alternative.

Indeed, for \eqref{eq:zero} with $n\geq 2$, consider the operator $T\colon\Hilbert^{n-1}\to\Hilbert^{n-1}$ given by
\begin{equation}\label{eq:ryu extension}
T(\bz)
=  \bz +
\gamma\sqrt{\frac{2}{n-1}}\begin{pmatrix}
x_n-x_1 \\
x_n-x_2 \\
\vdots \\
x_{n}-x_{n-1} \\
\end{pmatrix}
\end{equation}
where $x_1,\dots,x_n\in \Hilbert$ are given by
\begin{equation*}
 \begin{cases}
    x_i=J_{F_i}\bigl(\sqrt{\frac{2}{n-1}}\,z_i+\frac{2}{n-1}\displaystyle\sum_{j=1}^{i-1}x_i\bigr) \qquad \forall i\in\{1,\dots,n-1\} \\
    x_n=J_{F_n}\bigl(\frac{2}{n-1}{\displaystyle\sum_{j=1}^{n-1}x_i}-\sqrt{\frac{2}{n-1}}\displaystyle\sum_{i=1}^{n-1}z_i\bigr). 
   \end{cases}
\end{equation*}
This has coefficient matrices $M\in\mathbb{R}^{(n-1)\times n},S\in\mathbb{R}^{n\times(n-1)}$ and $L\in\mathbb{R}^{n\times n}$ are given by
 $$ M = -S^* = \sqrt{\frac{2}{n-1}}\begin{bmatrix}
                -1 &  0 & \dots & 0 & 1 \\
                 0 & -1 & \dots & 0 & 1 \\                
                 \vdots & \vdots & \ddots & \vdots & \vdots  \\
                 0 &  0 & \dots & -1 & 1 \\  
               \end{bmatrix},
               \qquad
    N=\frac{2}{n-1}\begin{bmatrix}
                0 & 0 & \dots & 0 & 0 \\
                1 & 0 & \dots & 0 & 0 \\
                1 & 1 & \dots & 0 & 0 \\
                \vdots & \vdots &\ddots &\vdots & \vdots \\
                1 & 1 & \dots & 1 &  0 \\
                \end{bmatrix}. $$
Note that, when $n=3$, these matrices are precisely those given in \eqref{eq:ryu3} of Example~\ref{ex:ryu3}. Consequently, the setting considered in this example provides a one parameter family that coincides with Ryu splitting in the special case when $n=3$.
                
It is straightforward to verify that Assumption~\ref{ass:resolvent}\eqref{ass:res_a}-\eqref{ass:res_c} holds. To verify Assumption~\ref{ass:resolvent}\eqref{ass:res_d},  observe that
 $$ M^*M + N+N^*-2\Id =\frac{2}{n-1}\begin{bmatrix}
               2-n & 1 & \dots & 1 & 0 \\
               1 & 2-n & \dots & 1 & 0 \\
               \vdots & \vdots & \ddots & \vdots & \vdots  \\
               1 & 1 & \dots & 2-n & 0 \\
               0 & 0 & \dots & 0 & 0 \\
               \end{bmatrix} $$   
This matrix is diagonally dominate with non-positive diagonal entries, hence it is negative definite.
\end{example}

Before our next example, we first recall some notational from graph theory (see, for instance, \cite{chung1996lectures}). Consider an (undirected) graph $G=(V,E)$ with vertex set $V=\{v_1,\dots,v_n\}$ and edge set $E=\{e_1,\dots,e_m\}$. The \emph{degree matrix} of $G$ is the diagonal matrix $D\in\mathbb{R}^{n\times n}$ such that $D_{ii}$ is equal to the degree of $v_i$. The \emph{adjacency matrix} of $G$ is the matrix $A\in\{0,1\}^{n\times n}$ such that $A_{ij}=1$ if only if $v_i$ and $v_j$ are adjacent. The \emph{graph Laplacian} is the matrix $L\in\mathbb{R}^{|V|\times|V|}$ given by $L=D-A$. An \emph{orientation} of an undirected graph is a directed graph obtained by assigning a direction to each edge in the original graph. The \emph{incidence matrix} of a directed graph $G=(V,E)$ is the matrix $B\in\{0,\pm 1\}^{|V|\times|E|}$ given by
\begin{equation}\label{eq:incid}
  B_{ij} = \begin{cases}
              -1 & \text{if }e_j\text{ leaves vertex }v_i, \\
              +1 & \text{if }e_j\text{ enters vertex }v_i, \\
               0 & \text{otherwise.}
            \end{cases} 
\end{equation}            
Let $B_\sigma$ denote an oriented incidence matrix of an undirected graph $G$, that is, an incidence matrix for any orientation of $G$. Then its graph Laplacian can also satisfies $\mathcal{L}=B_\sigma B_\sigma^T.$

\begin{example}[$d$-regular networks]\label{ex:regular}
Consider a simple, connected $d$-regular graph $G=(V,E)$ with ${|V|=n}$. As a consequence of the Handshake Lemma, $G$ necessarily has $|E|=\frac{nd}{2}$ edges. Let $B_\sigma$ be an oriented incidence matrix for $G$. Consider the coefficient matrices given by 
$$-S^*=M=\sqrt{\frac{2}{d}}B^T_\sigma\in\mathbb{R}^{|E|\times|V|}$$
and $N\in\mathbb{R}^{|V|\times |V|}$ being the lower triangular matrix (with zero diagonal) satisfying $N+N^*=\frac{2}{d}A$. We claim that these matrices satisfy Assumptions~\ref{ass:resolvent}. Indeed, we have:
\begin{enumerate}[(a)]
\item By \eqref{eq:incid}, we have $\ker M=\ker B_\sigma^T = \mathbb{R}e$ where $e=(1,\dots,1)^T\in\mathbb{R}^{|E|}$.
\item Since $\sum_{i,j=1}^nA_{ij}=2|E|$, it follows that $\sum_{i,j=1}^nN_{ij}=\frac{1}{d}\sum_{i,j=1}^nA_{ij} = \frac{2|E|}{d} = n$.
\item $S^*=-M$ holds by definition.
\item Since $G$ is $d$-regular, $L=B_\sigma B_\sigma^T=d\Id-A$. Hence 
$$M^*M+N+N^*-2\Id = \frac{2}{d}(d\Id-A) +\frac{2}{d}A-2\Id = 0\preceq 0.$$
\end{enumerate}
Altogether, we have that the operator $T\colon\Hilbert^{|E|}\to\Hilbert^{|E|}$ given by 
\begin{equation}\label{eq:d reg T}
 T(\bz) = \bz +\gamma\bM\bx,\quad \bx=J_{\bF}(\bS\bz+\bN\bx) 
\end{equation}
satisfies the conditions of Section~\ref{s:resolvent splitting}. It is interesting to interpret the variables in the context of the above example. Note that $\bz\in\Hilbert^{|E|}$ and $\bx\in\Hilbert^{|V|}$, so that $\bz$ represents edges of $G$ and $\bx$ represents vertices of $G$. The matrix $M$ describes information flow from vertices to their incident edges. Similarly, the matrix $S$ describes information flow from edges to their adjacent vertices. Finally, $N$ describes a direct information flow between adjacent vertices. 

Since $|E|$ is potentially large, it is often better to avoid using the operator $T$ directly. Using the observation outlined in Remark~\ref{r:implementation} and the fact that $SM=-M^*M=-L$,  \eqref{eq:d reg T} can be rewritten in terms of the operator $\widetilde{T}\colon\Hilbert^{|V|}\to\Hilbert^{|V|}$ given by
\begin{equation}\label{eq:alg regular}
 \widetilde{T}(\bv) = \bv - \gamma\frac{2}{d}\bL\bx,\quad \bx=J_{\bF}(\bv+\bN\bx), 
\end{equation}
This is suited for a distributed implementation with node $i$ responsible for computing $v_i$ and $x_i$. The corresponding iteration for \eqref{eq:alg regular} is given explicitly in \eqref{eq:d reg full} (or \eqref{eq:d reg alg}).

 Note that the setup used in this example does not directly generalise to arbitrary network for two reasons. Firstly, (b) requires the number of edges in $G$ to be exactly $|E|=\frac{nd}{2}$, and secondly, the computation in (d) used the fact that the degree matrix is given by $D=d\Id$.
\end{example}

\section{Numerical Experiments}\label{s:experiments}
In this section, we demonstrate the algorithms presented in Section~\ref{s:realisations} on a toy problem from statistics. Given a finite dataset $C:=\{c_1,\dots,c_n\}\subseteq\mathbb{R}$, we consider finding a point $x\in\mathbb{R}$ which minimises dispersion to $C$ as measured by the $\ell_p$-norm ($p\geq 1$). That is, we consider
\begin{equation}\label{eq:lp}
 \min_{x\in\mathbb{R}}\left(\frac{1}{n}\sum_{i=1}^n|x-c_i|^p\right)^{1/p}.
\end{equation}
The optimal value of \eqref{eq:lp} represents a measure of \emph{dispersion} for the dataset $C$, and solutions of \eqref{eq:lp} represent are measures of \emph{central tendency}. Important examples of measures of dispersion/central tendency pairs within this framework include: \emph{average absolute deviation/median} when $p=1$, \emph{standard deviation/mean} when $p=2$, and \emph{maximum deviation/midrange} when $p=+\infty$ (which is interpreted via the limit for $p\to+\infty$ in~\eqref{eq:lp}).

For the purpose of this demonstration, we focus on the setting where $p=1$, which was previously considered in \cite{malitsky2021resolvent}. In this case, the sub-differential sum rule \cite[Corollary~16.48]{bauschke2017convex} implies that solutions to~\eqref{eq:lp} can be characterised by the monotone inclusion
$$ 0 \in \sum_{i=1}^nA_i(x)\subseteq\mathbb{R}\text{~~where~~}A_i:=\partial|\cdot-c_i|. $$
In all our experiments, problem instances are obtained by generating a vector $c\in\mathbb{R}^n$ entrywise through sampling the standard normal distribution.

\subsection{Splitting algorithms with minimal lifting}
Our first experiment is a comparison between frugal resolvent splitting algorithms for \eqref{eq:mono} with \emph{minimal lifting}. That is, algorithms for which the underlying fixed point operator is defined on the space~$\Hilbert^{n-1}$. For further details on minimal lifting, the reader is referred to \cite{ryu2020uniqueness,malitsky2021resolvent}. In particular, we compare the ``resolvent splitting with minimal lifting'' due to Malitsky \& Tam \cite{malitsky2021resolvent} (see also Example~\ref{ex:minimal}) with our proposed extension of Ryu's splitting algorithm~\cite{ryu2020uniqueness,aragon2021strengthened} from Example~\ref{ex:ryu new}.  

Figure~\ref{f:ex11} shows the effect of changing the parameter $\gamma\in(0,1)$ on the decay of the fixed point residual, as measured by the quantity $\|(\Id-T)(\bz^k)\|/\gamma$ where $(\bz^k)$ denotes sequence generated by the algorithm, for a problem with $n=10$. In both cases, the initial point $\bz^0$ was taken as the vector of all zeros. These results suggest that both algorithms perform better with larger values of~$\gamma$ for this scenario.

\begin{figure}[!htb]
\centering
\includegraphics[width=0.47\textwidth]{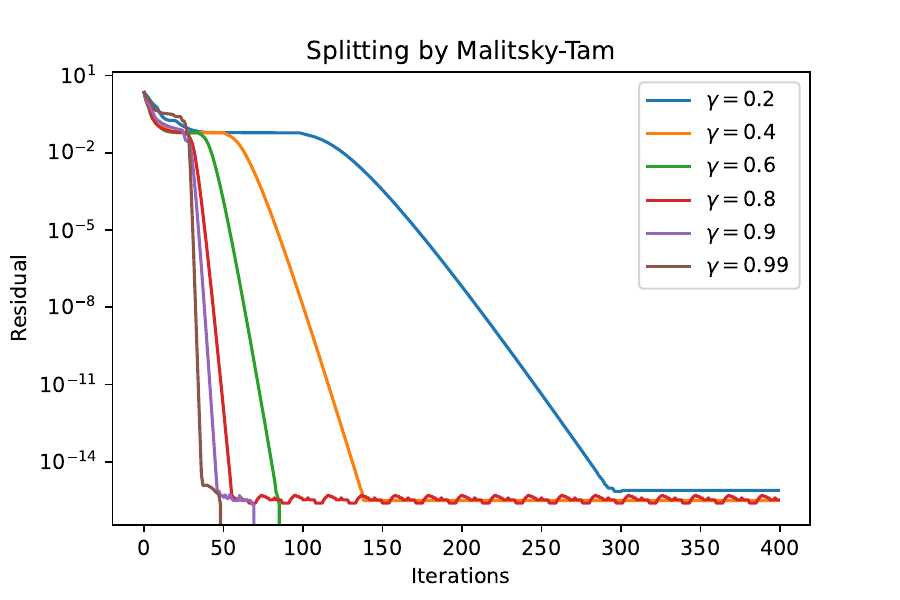}
\includegraphics[width=0.47\textwidth]{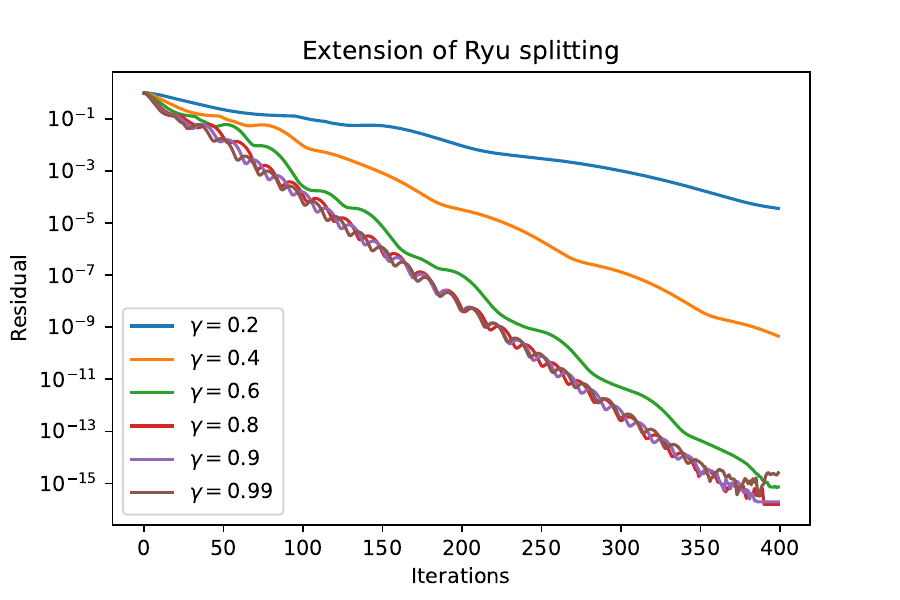}
\caption{The effect of varying the parameter $\gamma\in(0,1)$ for the splitting algorithm by Malitsky--Tam (left) and the proposed extension of Ryu's splitting algorithm (right).\label{f:ex11}}
\end{figure}

Figure~\ref{f:ex12} shows results for datasets of size $n=100$ and $n=250$. For both algorithms, $\gamma=0.99$ was used as was best in Figure~\ref{f:ex11}. Figure~\ref{f:ex12} suggest that the splitting algorithm by Malitsky--Tam converges faster,  although the proposed extension of Ryu's algorithm is better when only a relatively fewer number of iteration can be computed.

\begin{figure}[!htb]
\centering
\includegraphics[width=0.47\textwidth]{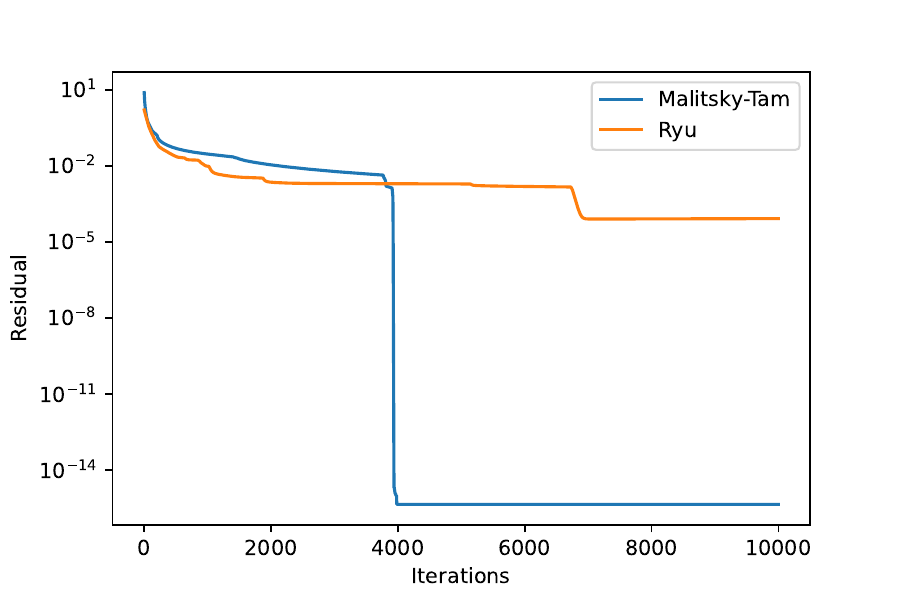}
\includegraphics[width=0.47\textwidth]{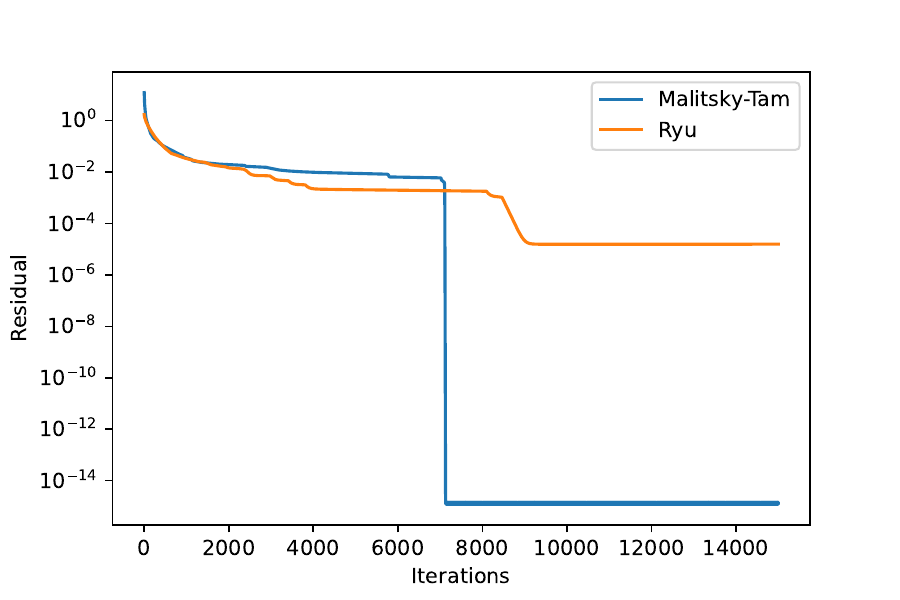}
\caption{Results with $\gamma=0.99$ for problems with $n=100$ (left) and $n=250$ (right).\label{f:ex12}} 
\end{figure}

\subsection{Distributed splitting on regular networks}
Our second experiment is a comparison between decentralised resolvent splitting algorithms for \eqref{eq:mono} on $d$-regular networks. In particular, we compare the algorithm proposed in Example~\ref{ex:regular} with \emph{PDHG}~\cite{chambolle2011first,condat2013primal} (as discussed in \cite[Section~5.1]{malitsky2021resolvent}) and \emph{P-EXTRA}~\cite{shi2015proximal}. Our experiments are performed using four $d$-regular networks with $n=11$ vertices and $d\in\{2,4,6,8\}$, as shown in Figure~\ref{f:networks}.

\begin{figure}[!htb]
\centering
\includegraphics[width=0.6\textwidth]{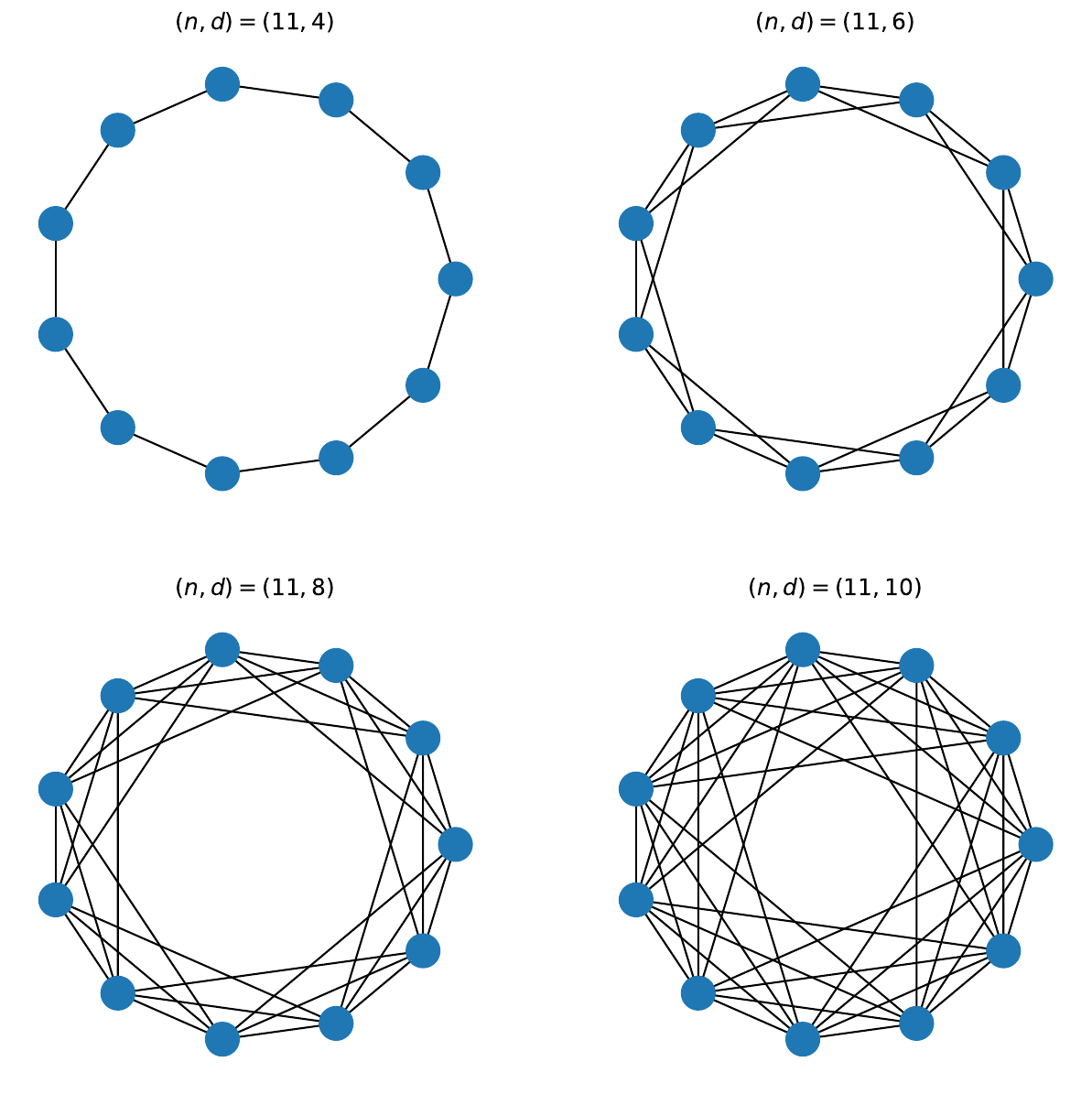}
\caption{The four regular networks used in the experiments.\label{f:networks}}
\end{figure}

Parameters in the proposed method and PDHG where chosen by trial and error for best performance; for the proposed method, $\gamma=0.5$ was used, and for the PDHG, $(\tau,\sigma)=(\frac{1}{10\|L\|^{1/2}},\frac{10}{\|L\|^{1/2}})$ was used. For P-EXTRA, the mixing matrix $W$ as taken to be \emph{Laplacian-based constant edge weight matrix}, as suggested in \cite[Section~2.4]{shi2015extra}. 

Figure~\ref{f:dis res} reports the distance of the current iterate to the solution for the three methods, on each of the networks in Figure~\ref{f:networks}. The figure suggests favourable performance of the proposed method. However, this might be expected as the proposed method is more specialised --- it can only be used on regular networks, whereas the other methods can be applied can used on arbitrary connected networks. Overall, all methods were observed to converge faster as the network connectivity increased, as expected.

\begin{figure}[!htb]
\centering
\includegraphics[width=0.47\textwidth]{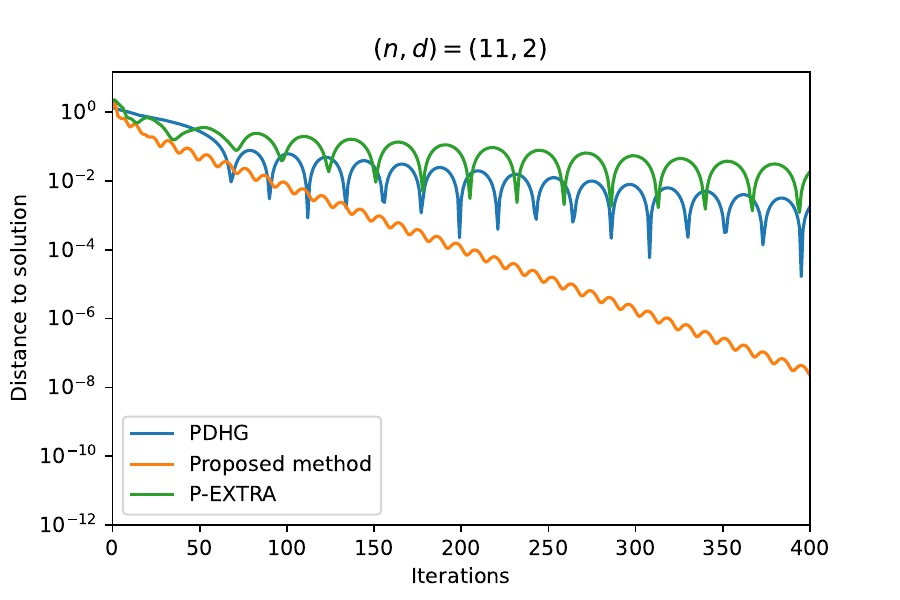}
\includegraphics[width=0.47\textwidth]{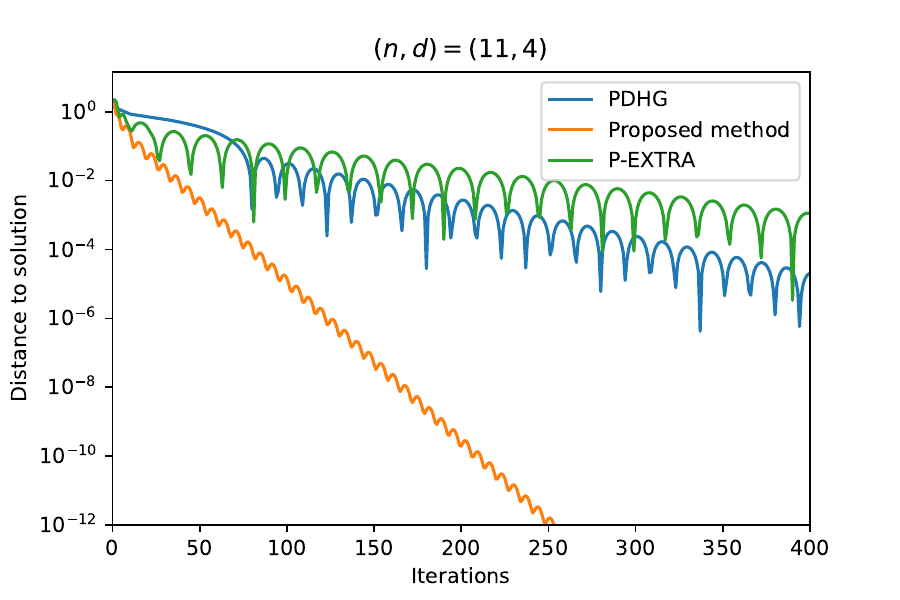}
\includegraphics[width=0.47\textwidth]{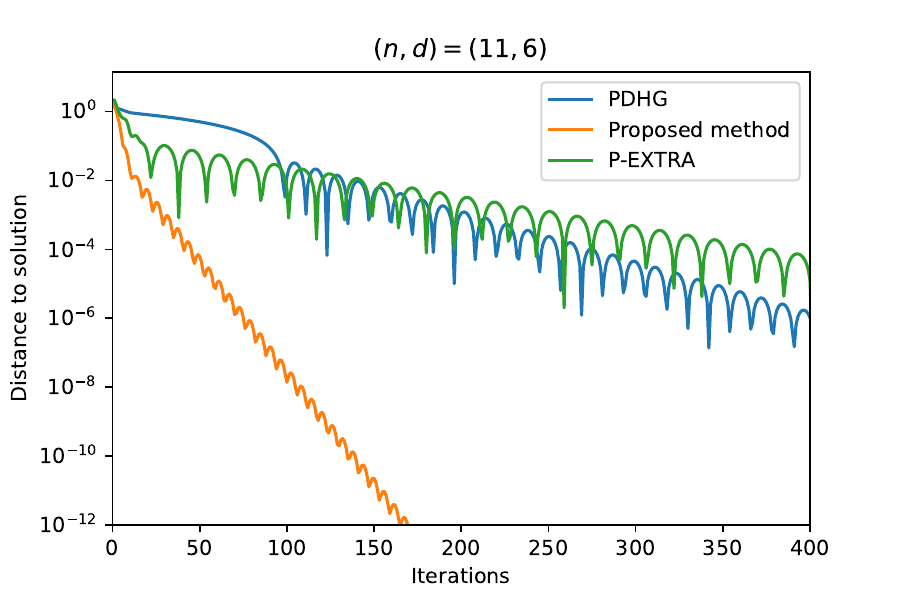}
\includegraphics[width=0.47\textwidth]{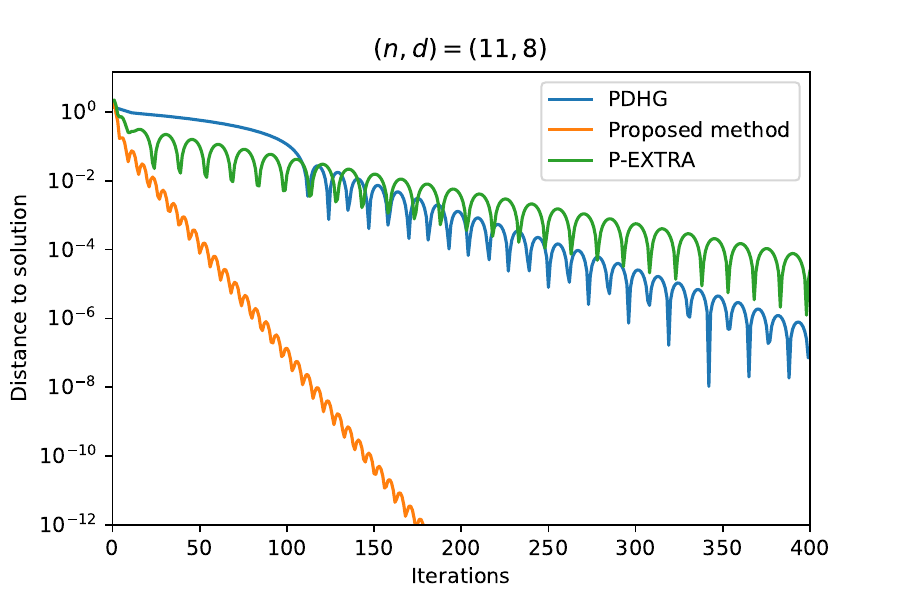}
\caption{Results for the four regular networks shown in Figure~\ref{f:networks}.\label{f:dis res}} 
\end{figure}

\section*{Acknowledgements}
The author would like to thank the associate editor, anonymous referee and Yura Malitsky for their comments which improved the manuscript. MKT is supported in part by Australian Research Council grants DE200100063 and DP230101749.

\bibliographystyle{abbrv} 
\bibliography{biblio}

\end{document}